\documentclass[12pt,a4paper]{amsart}
\usepackage{latexsym}
\usepackage{amssymb}
\usepackage{enumerate}
\usepackage{amscd}
\usepackage{tikz-cd}
\usepackage[all]{xy}
\usepackage[utf8]{inputenc}
\usepackage[mathscr]{eucal}
\theoremstyle{plain}
\newtheorem{prob}{Problem}[section]
\newtheorem{defin}[prob]{Definition}
\newtheorem{theorem}[prob]{Theorem}

\newtheorem*{theorem*}{Theorem}
\newtheorem*{question*}{Question}
\newtheorem{prop}[prob]{Proposition}
\newtheorem{corollary}[prob]{Corollary}
\newtheorem{lemma}[prob]{Lemma}
\theoremstyle{remark}
\newtheorem{example}[prob]{Example}
\newtheorem{remark}[prob]{Remark}

\newcommand{\R}{\mathbb{R}}

\newcommand{\adef}{\begin{defin}}
\newcommand{\zdef}{\end{defin}}

\def\Ext{\operatorname{Ext}}

\def\PB{\operatorname{PB}}
\def\PO{\operatorname{PO}}

\newcommand{\ra}{\rangle}
\newcommand{\la}{\langle}
\newcommand{\ult}{\textrm{ult}}

\newcommand{\con}{\mathfrak c}
\renewcommand\c[1]{\mathcal{#1}}
\renewcommand\b[1]{\mathbb{#1}}
\newcommand{\fA}{\mathfrak A}
\newcommand{\fB}{\mathfrak B}
\newcommand{\bA}{\mathbb A}
\newcommand{\bS}{\mathbb S}
\newcommand{\cA}{\mathcal A}
\newcommand{\cB}{\mathcal B}

\newcommand{\cF}{\mathcal F}
\newcommand{\vf}{\varphi}

\newcommand{\fin}{\rm fin}

\newcommand{\arrow}[1]{\overset{#1}{\longrightarrow}}
\newcommand{\shortseq}[5]{0\arrow{}{#1}\arrow{#4}{#2}\arrow{#5}{#3}\arrow{} 0}

\def\Ext{\operatorname{Ext}}

\newcommand{\K}{\mathcal K}
\newcommand{\cL}{\mathcal L}
\newcommand{\sub}{\subseteq}
\newcommand{\er}{\mathbb R}
\newcommand{\mi}{\protect{\sf ri}}
\newcommand{\wt}{\widetilde}

\title{Sailing over three problems of Koszmider}

\author[Cabello]{F\'elix Cabello S\'anchez}
\address{Instituto de Matem\'aticas Imuex\\ Universidad de Extremadura\\
Avenida de Elvas\\ 06071-Badajoz\\ Spain} \email{fcabello@unex.es\\
castillo@unex.es\\salgueroalarcon@unex.es}
\author[Castillo]{Jes\'{u}s M.F. Castillo}
\author[Marciszewski]{Witold Marciszewski}
\address{Institute of Mathematics\\
	University of Warsaw\\ Banacha 2\newline 02--097 Warszawa\\
	Poland} \email{wmarcisz@mimuw.edu.pl}
\author[Plebanek]{Grzegorz Plebanek}
\address{Mathematical Institute \\ University of Wroc\l aw\\
	pl.Grunwaldzki, 2/4\\ 50-384 Wroclaw \\ Poland\\} \email{grzes@math.uni.wroc.pl}
\author[Salguero-Alarc\'on]{Alberto Salguero-Alarc\'on}

\thanks{The research of the first and second authors has been supported in part by Project MINCIN MTM2016-76958-C2-1-P. The research of the first, second and fifth authors has been supported in part by Project IB16056 de la Junta de Extremadura.
The fourth author has been supported by the grant 2018/29/B/ST1/00223 from National Science Centre, Poland.
The fifth author benefited from a grant associated to Project IB16056 de la Junta de Extremadura.}


\begin{document}

\begin{abstract}
We discuss  three problems of Koszmider on  the structure of the spaces of continuous functions on the Stone compact $K_{\c A}$ generated by an almost disjoint family $\c A$ of infinite subsets of $\omega$ --- we present a solution to two problems and develop  a previous results of Marciszewski and Pol answering the third one.
We will show, in particular,  that assuming Martin's axiom  the space $C(K_{\c A})$ is uniquely determined up to isomorphism by the cardinality of $\c A$
whenever $|\c A|<\mathfrak c$,   while there are $2^{\mathfrak c}$ nonisomorphic spaces $C(K_{\c A})$ with  $|\c A|= \mathfrak c$. We also investigate Koszmider's problems in the context of the class of separable Rosenthal compacta and indicate the meaning of our results in the language of twisted sums of $c_0$ and some $C(K)$ spaces.
\end{abstract}

\maketitle
\section{Introduction}

Koszmider poses in \cite{kosz} five problems about the structure of the spaces of continuous functions on the Stone compact $K_{\c A}$ generated by an almost disjoint family $\c A$ of infinite subsets of $\omega$. Problem 2, that Koszmider himself solves, is the existence of an almost disjoint family $\mathcal A$ such that under either the Continuum Hypothesis {\sf CH} or Martin's Axiom {\sf MA},
$$C(K_{\mathcal  A})\simeq c_0 \oplus C(K_{\mathcal A})$$
 is the only possible decomposition of $C(K_{\mathcal A})$ in two infinite dimensional subspaces. Very recently Koszmider and Laustsen have obtained in \cite{KL20}  the same result without any additional set-theoretic assumptions. Problem 1 asks whether a similar separable space exists. Argyros and Raikoftsalis \cite{argraik} call a Banach space $\mathcal X$ \emph{quasi-prime} if there exists an infinite dimensional subspace $\mathcal Y$ such that $\mathcal X \simeq \mathcal Y \oplus \mathcal X$ is the only possible nontrivial decomposition of $\mathcal X$. If, moreover $\mathcal Y$ is not isomorphic to $\mathcal X$ then $\mathcal X$ is called \emph{strictly quasi-prime}. Argyros and Raikoftsalis \cite{argraik} show the existence, for each $\ell_p$, $p\geq 1$, (resp. $c_0$) of a separable strictly quasi-prime space $\mathcal X_p\simeq \ell_p\oplus \mathcal X_p$ (resp. $\mathcal X_0\simeq c_0\oplus \mathcal X_0$).

 We are concerned in this paper with the other three:
\begin{itemize}
\item \textbf{Problem 3} Assuming {\sf MA}, is it true that if $|\c A|=|\c B| < \mathfrak c$ then $C(K_{\c A}) \simeq C(K_{\c B})  $ ?

\item \textbf{Problem 4} Assuming {\sf MA} and $|\c A| < \mathfrak c$ is $C(K_{\c A}) \simeq C(K_{\c A}) \oplus C(K_{\c A})$?

 \item \textbf{Problem 5} Are there two almost disjoint families $\c A, \c B$ of the same cardinality such that $C(K_{\c A})$ is not isomorphic to $C(K_{\c B})$ ?
\end{itemize}

Koszmider's questions are  mentioned
by Hru\v{s}\'ak \cite[9.2]{Hr14} in his survey on applications of almost disjoint families.
Let us point out that  Problem 5 was actually solved by Marciszewski and Pol \cite{marcpol} who
gave an example of a pair of almost disjoint families $\c A, \c B$ of cardinality $\mathfrak c$ such that $C(K_{\c A})\not\simeq C(K_{\c B})$ --- this is a direct
consequence of \cite[Theorem 3.4]{marcpol}.
It was, moreover,  briefly outlined in \cite[7.4]{marcpol} how one can prove, using some  ideas from \cite{Ma88},
that there are $2^\con$ isomorphism types of Banach spaces of the form $C(K_{\c A})$.

We will solve affirmatively Problems 3 and 4 and present a detailed, self-contained solution to Problem 5.
Summing up, we will obtain

\begin{theorem}\label{abc} $\;$
\begin{enumerate}[(a)]
\item Under {\sf MA}($\kappa$),  if $\c A$ and $\c B$ are almost disjoint families such that  $|\cA|=|\cB|=\kappa$, then
     $C(K_{\c A})\simeq C(K_{\c B})$.
\item Under {\sf MA}($\kappa$), if $\cA$ is an almost disjoint family with $|\cA|=\kappa$,  then $C(K_{\c A}) \simeq C(K_{\c A})\oplus C(K_{\c A})$.
\item  There are $2^{\mathfrak c}$ nonisomorphic spaces $C(K_{\c A})$ for almost disjoint families $\cA$  of size    $\con$.
\end{enumerate}\end{theorem}

Here {\sf MA}($\kappa$) denotes Martin's axiom for $ccc$ partial orders and $\kappa$ many dense sets (recall that if we assume
{\sf MA}($\kappa$) then, automatically, $\kappa<\con$).

In the context of part (b) of the above theorem it is worth to recall that the first example of an infinite compact space $K$ such that $C(K)$ is not isomorphic to $C(K)\oplus C(K)$ was given by Semadeni in \cite{Se60}. His space $K$ was the interval of ordinals $[0,\omega_1]$ equipped with the order topology. In \cite{Ma1}  an example of an almost disjoint family $\c A$ was constructed, such that the space $C(K_{\c A})$ endowed with the weak topology is not homeomorphic to $C(K_{\c A})\oplus C(K_{\c A})$ with the weak topology, hence the spaces $C(K_{\c A})$ and $C(K_{\c A})\oplus C(K_{\c A})$ are not isomorphic.

Our proof of  Theorem \ref{abc}(c) follows a relatively simple counting argument suggested in \cite{marcpol}.
This approach is applicable to study a more general question, on the number of pairwise nonisomorphic twisted sums of $c_0$ and Banach spaces of the form $C(K)$. In Section 5 we apply the technique to $K$ being either  the one point compactification of the discrete space $\con$ or
 the classical nonmetrizable compact space called the double arrow space or the split interval .

Recall that Rosenthal compacta form an interesting and important class of topological spaces that originated in functional analysis, see the survey articles \cite{De14} and \cite{Mar03}. Sections 6 and 7
are devoted to study Rosenthal compacta in the light of our results from the first part of the paper.
Our main motivation here is related to the fact that the proof of \ref{abc}(c)  does not provide `concrete' examples of pairs of nonisomorphic Banach spaces in question. Therefore, we want to indicate more constructive ways of obtaining such examples.
Note that, if we construct an almost disjoint family $\c A$ in an `effective' way (i.e., such $\c A$, treated as subspace of the Cantor set $2^\omega$, is Borel or analytic), then the resulting space $K_{\c A}$ is a  Rosenthal compactum, see \cite[Lemma 4.4]{MaPo2}.
We show in Section 7, that using a completely different argument one can
name an uncountable sequence of almost disjoint families $\{ {\c A}_\xi:\xi<\omega_1\}$ such that every $K_{{\c A}_\xi}$ is a Rosenthal compactum
and the Banach spaces $C( K_{{\c A}_\xi})$ are pairwise nonisomorphic.
We also prove a similar result on the class of twisted sums of $c_0$ and $C(\bS)$, where $\bS$ is the double arrows space, a well-known nonmetrizable separable Rosenthal compactum itself.
The latter result partially extend \cite[Corollary 4.11]{avimar} stating that
$c_0$ admits a nontrivial twisted sum with $C(K)$ for every  nonmetrizable separable Rosenthal compactum $K$.
Our considerations presented in Section 7 build on delicate  descriptive properties of separable Rosenthal compacta;
to make the presentation reasonably self-contained, we provide in Section 6  all the relevant facts to be used.

The authors are very grateful to the referee for a very careful reading and several comments that enabled us to improve the
presentation.

\section{Preliminaries on twisted sums of Banach spaces}
We will write $A\simeq B$ to mean that the Banach spaces $A$ and $B$ are isomorphic. An \emph{exact sequence} $\texttt z$ of Banach spaces is a diagram
$$
\begin{CD} 0@>>> Y @>\jmath>> Z @>\rho>>
X@>>>0
\end{CD}$$
formed by Banach spaces and linear continuous operators in which the kernel of each arrow coincides with the image of the
preceding one. The middle space $Z$ is usually called a \emph{twisted sum} of $Y$
and $X$. By the open mapping theorem, $Y$ must be isomorphic to a subspace of $Z$ and $X$ to the quotient $Z/Y$. Two exact sequences
$\texttt z$ and $\texttt s$ are said to be equivalent, denoted $\texttt z \equiv \texttt s$, if there is an operator $T: Z\to S$ making commutative the diagram
$$
\begin{CD} 0@>>> Y @>\jmath_{\texttt z}>> Z @>\rho_{\texttt z}>> X@>>>0\\
&&@| @VVTV @|\\
 0@>>> Y @>>\jmath_{\texttt s}> S @>>\rho_{\texttt s}> X@>>>0
\end{CD}$$
The sequence $\texttt z$ is said to be \emph{trivial}, or \emph{to split}, if the injection $\jmath$ admits a left inverse; i.e., there is a linear continuous projection $P: Z \to Y$ along $\jmath$. Equivalently, if $\texttt z \equiv 0$ where $0$ denotes the exact sequence
$0\to Y \to Y\oplus X \to X \to 0$. Given an exact sequence $\texttt z$ and an operator $\gamma: X'\to X$ the \emph{pull-back} exact sequence $\texttt z \gamma$ is the lower sequence in the diagram
$$
\begin{CD} 0@>>> Y @>\jmath>> Z @>\rho>> X@>>>0\\
&&@| @AAA @AA\gamma A\\
0@>>> Y @>>> \PB @>>> X'@>>>0
\end{CD}$$
where $\PB = \{(z,x')\in Z\oplus_\infty X': \rho z = \gamma x'\}$ endowed with the subspace norm. Dually, given an operator $\alpha: Y\to Y'$ the \emph{push-out} exact sequence $\alpha \texttt z $ is the lower sequence in the diagram
$$
\begin{CD} 0@>>> Y @>\jmath>> Z @>\rho>> X@>>>0\\
&&@V\alpha VV  @VVV @| \\
0@>>> Y' @>>>\PO @>>> X@>>>0
\end{CD}$$
where $\PO = (Y'\oplus_1 Z)/\Delta$, where $\Delta = \{(\alpha y, - \jmath y): y\in Y\}$ endowed with the quotient norm. We need to keep in mind the following two facts:

\begin{prop}\label{semi} Let $\texttt z$ and $\texttt s$ be two exact sequences $$\shortseq{Y}{Z}{X}{}{\rho_{\texttt z}} \qquad  \qquad \shortseq{Y}{S}{X}{}{\rho_{\texttt s}}$$
\begin{enumerate}[(1)]
	\item If $\texttt s \rho_{\texttt z}\equiv 0$ then there is an operator $\tau: Y \to Y$ such that $\tau \texttt z \equiv \texttt s$.
	\item If $\texttt s \rho_{\texttt z} \equiv 0 \equiv \texttt z \rho_{\texttt s} $ then $Y\oplus S \simeq Y\oplus Z$.
\end{enumerate}
\end{prop}
\begin{proof}
	The first fact can be seen in \cite{cabecastlong} and is an easy consequence of the homology sequence; the second is the so-called diagonal principle \cite{castmoreisr}.
\end{proof}

The space of exact sequences between two given spaces $Y$ and $X$ modulo equivalence will be denoted $\Ext(X,Y)$ (in reverse order). We will write
$\Ext(X,Y)=0$ to mean that all elements $\texttt z\in \Ext(X,Y)$ are $\texttt z \equiv 0$.

In this paper we are mainly concerned with twisted sums of $c_0$ and $C(K)$, where $K$ is a compact space. These objects have a long tradition; for instance, if $\c A$ is an almost disjoint family of cardinality $\mathfrak c$, a kind of $\ell_p$ version of $C(K_{\c A})$ was introduced in \cite{johnlind}, and has been sometimes called a Johnson-Lindenstrauss space \cite[Theorem 4.10.a]{castgonz} or \cite[3.1]{ephe}; even if the standard practice is to call Johnson-Lindenstrauss space to the true $\ell_p$-versions. The classical Sobczyk theorem, $c_0$ is complemented in any separable Banach superspace, implies that $\Ext(C(K), c_0)=0$ whenever $K$ is metrizable. The question of whether
$\Ext(C(K), c_0)\neq 0$ for every non-metrizable $K$ was posed in \cite{ccy,ccky} and has been intensively studied \cite{Ca16,CT16,maple,avimar,Co18,CT18}. While for some classes of compacta $K$ one can demonstrate that $\Ext(C(K), c_0)\neq 0$ more or less effectively, the problem cannot be decided within the usual set theory.
Indeed, one has

\begin{theorem}\label{avimar}$\;$
	\begin{enumerate}[(a)]
	\item Assuming {\sf CH},   $\Ext(C(K),c_0) \neq 0$ for every non-metrizable compact space $K$ (see
\cite[Theorem 5.8]{avimar}).
	\item Assuming {\sf MA}$(\kappa)$,  $\Ext(C(K_{\c A}), c_0)=0$ whenever  $|\c A|=\kappa$
(\cite[Corollary 5.3]{maple}).
\item Assuming {\sf MA}$(\kappa)$,  $\Ext(C(K), c_0)=0$ whenever  $K$ is a separable scattered space of weight $\kappa$ and finite height
(this extension of (b) is due to Correa and Tausk  \cite[Corollary 4.2]{CT18}).
	\end{enumerate}
\end{theorem}

\section{Preliminaries on compacta} \label{prel}

The most natural method of constructing twisted sums of $c_0$ and $C(K)$ is to consider a
 compact space $L$ of the form $L=K\cup\omega$,  containing  $K$ as a subspace  and additionally containing a countable infinite set of isolated points denoted simply by $\omega$ (we tacitly assume that $K\cap\omega=\emptyset$).
Such a space $L$ will be called a {\em countable discrete extension} of $K$.

Note that if $L$ is a countable discrete extension of $K$ then the subspace $Z=\{g\in C(L): g|K\equiv 0\}$  of $C(L)$
is a natural copy of $c_0$, whereas $C(L)/Z$ may be identified with $C(K)$.
This was fully discussed in \cite{maple} and \cite{avimar}. The following useful observation is a particular case of
 \cite[Theorem 2.8(a)]{maple} or \cite[Theorem 4.13]{avimar}.

\begin{theorem}\label{i:1}
	Suppose that $L=K\cup\omega$ is a countable discrete extension of a compact space $K$ such that $\omega$ is dense in $L$ (in other words,
	$L$ is a compactification of $\omega$ whose remainder is homeomorphic to $K$).
	If $K$ is not $ccc$, then $C(L)$ is a nontrivial twisted sum of $c_0$ and $C(K)$.
\end{theorem}

Our solution to Problem 5 is motivated by examining the variety of twisted sums of $c_0$ and $C(K)$, where $K$ is either the Stone space constructed from an almost disjoint family of subsets of $\omega$, or $K=\b S$, where $\b S$ denotes the classical double arrow space (see below).

\subsection{The Stone compact of an almost disjoint family} Recall that a topological space is scattered if every of its subsets contains an isolated point
 (see \cite{So03} for the basic properties and further references). We consider here some scattered compact spaces $K$; $K'$ is the first derivative (the set
 of nonisolated points in $K$). Higher Cantor-Bendixon derivatives  $K^{(\alpha)}$ are defined inductively; the height of $K$ is the first ordinal number $\alpha$ for which $K^{(\alpha)}=\emptyset$.

Recall that a family $\c A$ of infinite subsets of $\omega$ is almost disjoint if $A\cap B$ is finite for any distinct $A,B\in \c A$.
To every almost disjoint family $\c A$ one can associate a certain scattered compact space $K_\cA$ of height 3.
The space $K_\cA$ may be defined as
the Stone space of the algebra of subsets of $\omega$ generated by $\c A$ and all finite sets; alternatively,
\[K_\c A=\omega\cup \{A:A\in\c A\}\cup\{\infty\},\]
where points in $\omega$ are isolated, basic open neighborhoods of a given point $A\in\cA$ are of the form $\{A\}\cup (A\setminus  F)$ with $F\subseteq \omega$ finite, and $K_\c A$ is the one point compactification of
the locally compact space $\omega\cup \c A$. This means that $K_\cA$ is formed by adding a single point $\infty\notin\omega\cup \c A$
and declaring that the local base at $\infty$ consists of sets of the form
$\{\infty\}\cup V$, where $V$ is an open subset of $\omega\cup \c A$ with  compact complement.

We write $\bA(\kappa)$ for the (Aleksandrov) one-point compactification
of a discrete space of cardinality $\kappa$.  The space $K_\c A$ might be called the Aleksandrov-Urysohn compactum associated to
 an almost disjoint family $\c A$; this terminology was used in \cite{marcpol}, \cite{maple} and \cite{avimar}; see also \cite[section 2]{DV10}
(most often, spaces of the form $K_\c A$ are called  Mr\'owka spaces).
  Hru\v{s}\'ak \cite{Hr14} and Hern\'andez-Hern\'andez and  Hru\v{s}\'ak \cite{HH18}
  survey  various applications of that construction to topology and functional analysis.

Given a cardinal number $\kappa$, $c_0(\kappa)$ stands for the Banach space of all functions $f:\kappa\to\er$ having the property
that the set $\{\alpha<\kappa: |f(\alpha)|\ge \varepsilon\}$ is finite for every positive $\varepsilon$ (with the usual supremum norm).
Note that $c_0(\kappa)$ may be seen as a hyperplane in $C(\b A(\kappa))$ and that $c_0(\kappa)$ is isomorphic to $C(\b A(\kappa))$.
Moreover, if $|{\c A}|=\kappa$ then $K_\c A$ is a countable discrete extension of the space  $\bA(\kappa)$.
Hence, the following is a direct consequence of  Theorem \ref{i:1}.

\begin{theorem}\label{tcc:1}
Given any  almost disjoint family $\c A$ of cardinality $\kappa>\omega$, the space $C(K_\c A)$ is a nontrivial twisted sum of $c_0$ and $c_0(\kappa)$.
\end{theorem}


\subsection{The double arrow space.} \label{das}
Actually, we mention here a single compactum, the double arrow space, sometimes called the split interval, and later (in Sections 5 and 7) we shall consider the class of its countable discrete extensions.
We  denote this classical space by $\b S$; recall that
\[ \b S=\big((0,1]\times \{0\}\big)\cup \big([0,1)\times\{1\}\big),\]
is equipped with the order topology given by the
lexicographical order
$$ \la s,i\ra \prec \la t,j\ra  \mbox{ if either } s<t, \mbox{  or }
s=t \mbox{ and } i<j.$$

The space $\b S$ is a nonmetrizable separable compactum having a countable local base at every point.
It will be convenient to see $\b S$ as the Stone space of some algebra of subsets of a countable set.
Let $Q$ be a countable dense subset of $(0,1)$, and for each $x\in (0,1)$ put $P_x=\{q\in Q: q\le x\}$.
Let $\mathfrak A$ be the algebra of subsets of $Q$ generated by the chain $\{ P_x: x\in (0,1)\}$.
Then the Stone space $\ult(\mathfrak A)$, of all ultrafilters on the algebra $\mathfrak A$, is homeomorphic to $\b S$.

To see this, note first that every $\c F\in\ult(\mathfrak A)$ is uniquely determined by the set
\[ I(\c F)=\{x\in (0,1): P_x\in\c F\},\]
which is a subinterval
of $(0,1)$ of the form $[y,1)$ or $(y,1)$ for some $y\in [0,1]$.
Put $\c F=\c F_y^+$ if $I(\c F)=[y,1)$ and $\c F=\c F_y^-$ if $I(\c F)=(y,1)$, and define $h:\b S\to \ult(\mathfrak A)$ by

\begin{equation*}
h(y,i) =\begin{cases}
\c F_y^- & \text{ if } y\in [0,1), i=1;\\
\c F_y^+& \text{ if } y\in (0,1], i=0.

\end{cases}
\end{equation*}
Then $h$ is an homeomorphism of $\b S$ and $\ult(\mathfrak A)$.

\section{Solution to Problems 3 and 4}

 We start by recalling the following well-known observation:

 \begin{remark}\label{observation}
 If $K$ is an  infinite scattered compact space then $K$ contains a nontrivial convergent sequence, and so  $C(K)\simeq c_0\oplus C(K)$.
\end{remark}

\begin{proof}[Proof of Theorem \ref{abc} (a) and (b)] Pick two almost disjoint families $\c A, \c B $ of subsets of $\omega$ of cardinality $\kappa<\mathfrak c$.
Under {\sf MA}($\kappa$), we have
$$\Ext(C(K_{\c A}), c_0)=0=\Ext(C(K_{\c B}), c_0),$$
thanks to Theorem \ref{avimar} (b). Hence, if we call $\texttt a$ (resp. $\texttt  b$) the two exact sequences in the diagram
$$\begin{CD}
0@>>> c_0 @>>> C(K_{\c A}) @>{\rho_{\c A}}>> c_0(\kappa)@>>> 0\\
&&&&&&@|\\
0@>>> c_0 @>>> C(K_{\c B}) @>>{\rho_{\c B}}> c_0(\kappa)@>>> 0\\
\end{CD}$$
then $\texttt a \rho_{\c B}\equiv 0\equiv \texttt  b\rho_{\c A}$ and thus by Proposition \ref{semi}(2) and Remark \ref{observation}
$$C(K_{\c A})\simeq c_0 \oplus C(K_{\c A}) \simeq c_0\oplus C(K_{\c B}) \simeq  C(K_{\c B}).$$

Part (b) is consequence of the well-known algebraic identity
$$\Ext(C(K_{\c A})\oplus C(K_{\c A}), c_0)= \Ext(C(K_{\c A}), c_0) \times \Ext(C(K_{\c A}), c_0),$$
and the following\medskip

\noindent \textbf{Claim.} \emph{If $X$ is a twisted sum of $c_0$ and $c_0(\kappa)$ so that $\Ext(X, c_0)=0$ then $X\simeq C(K_{\c A})$}.\medskip

\noindent \emph{Proof of the Claim.} Assume the existence of an exact sequence
$$\begin{CD} 0@>>> c_0@>>> X @>\rho>> c_0(\kappa)@>>> 0\end{CD}$$

The same argument as before yields $X\oplus c_0 \simeq C(K_{\c A}) \oplus c_0 \simeq C(K_{\c A})$. All that is left to see is that the space $X$ has a complemented copy of $c_0$. Indeed, it follows from \cite{castsimo} that every twisted sum space $X$ as above has Pe\l czy\'nski's property (V). Therefore, there is a copy $X_0$ of $c_0$ in $X$ such that the restriction $\rho|_{X_0}$ is an isomorphism.  Since $\rho(X_0)$ must be necessarily complemented in $c_0(\kappa)$, $c_0$ will also be complemented in $X$. Hence $X\simeq c_0\oplus X$ and the proof concludes.\end{proof}

We do not know whether every twisted sum space
$$\begin{CD} 0@>>> c_0@>>> X @>\rho>> c_0(\kappa)@>>> 0\end{CD}$$
must be isomorphic to a $C(K)$-space. However,

\begin{prop} Let $\c A$ be an almost disjoint family of subsets of $\omega$ such that $|\c A|=\kappa<\mathfrak c$.
Under {\sf MA}($\kappa$),  every twisted sum space $X$
$$\begin{CD} 0@>>> c_0@>>> X @>>> c_0(\kappa)@>>> 0\end{CD}$$
is a quotient of $C(K_{\c A})$.
\end{prop}
\begin{proof} Since $\Ext(C(K_{\c A}), c_0)=0$ it follows from Proposition \ref{semi} (1) that there is a commutative diagram
$$\begin{CD}
0@>>> c_0 @>>> C(K_{\c A}) @>>> c_0(\kappa)@>>> 0\\
&&@VVV @VVV@|\\
0@>>> c_0 @>>> X @>>> c_0(\kappa)@>>> 0
\end{CD}$$

By the definition of the push-out space one has an exact sequence $$\begin{CD}
0@>>> c_0 @>>> c_0\oplus C(K_{\c A}) @>>> X@>>> 0
\end{CD}$$
Now, the obvious isomorphism $C(K_{\c A}) \simeq c_0\oplus C(K_{\c A})$ proves the assertion. \end{proof}

The following result contains a generalization of  Theorem \ref{abc}(a).

\begin{theorem}\label{ct:1} Assume ${\sf MA}(\kappa)$ and let $K_i$, $i= 0,1$ be separable scattered compacta of finite height and weight $\kappa$.
\begin{enumerate}
\item If $C(K_0')\simeq C(K_1')$ then $C(K_0)\simeq C(K_1)$.
\item If $C(K_0)\simeq C(K_1)$ then $X_0\simeq X_1$ whenever $X_i$, $i=0,1$, is a twisted sum $0\to c_0 \to X_i\to C(K_i)\to 0$.
\item Moreover, all iterated twisted sums of $X_{K_i}$, $i=0,1$ and $c_0$ are also isomorphic.
\end{enumerate}
\end{theorem}

\begin{proof} By Remark \ref{observation},  $C(K_i)\simeq c_0\oplus C(K_i)$. Moreover, one has exact sequences
$$\begin{CD}
0@>>>c_0 @>>> C(K_i) @>R>> C(K_i')@>>>0\end{CD}$$
where $R$ is the natural restriction map, which yields $\ker R\simeq c_0$. Consider the two exact sequences
$$\begin{CD}
0@>>> c_0 @>>> C(K_0) @>\alpha R>> C(K_1')@>>> 0\\
&&&&&&@|\\
0@>>> c_0 @>>> C(K_1) @>>R> C(K_1') @>>> 0
\end{CD}$$
where $\alpha: C(K_0')\to C(K_1')$ is an isomorphism. Let $\texttt z_1$ denote the lower sequence and
$\texttt z_0$ the upper sequence. Since $K_i$ is a separable scattered compact of finite height, Theorem \ref{avimar}(c)  implies
$\Ext(C(K_i), c_0)=0$, and thus $\texttt z_0 R\equiv 0\equiv \texttt z_1 \alpha \R$. Therefore,
$$C(K_0)\simeq  c_0\oplus C(K_0)\simeq c_0\oplus C(K_1)\simeq C(K_1),$$
and this proves (1). To prove (2), consider the two exact sequences $0\to c_0 \to X_i\stackrel{Q_i}\to C(K_i)\to 0$ and let $\alpha: C(K_0)\to C(K_1)$ be an isomorphism. Recall that $\Ext(X_i, c_0)=0$ by a 3-space argument (see \cite{cabecastlong}) and thus a similar reasoning as above can be used  with the two sequences
$$\begin{CD}
0@>>> c_0 @>>> X_0 @>\alpha Q_0 >> C(K_1)@>>> 0\\
&&&&&&@|\\
0@>>> c_0 @>>> X_1 @>>Q_1> C(K_1) @>>> 0
\end{CD}$$
The spaces $X_i$ obviously contain complemented copies of $c_0$ since $C(K_i)$ does so and the pull-back sequence
$$\begin{CD}
0@>>> c_0 @>>> X_i @> Q_i >> C(K_i)@>>> 0\\
&&@| @AAA @AAA\\
0@>>> c_0 @>>> c_0\oplus c_0 @>>> c_0 @>>> 0
\end{CD}$$
splits by Sobczyk's theorem. Therefore $X_i \simeq c_0\oplus X_i$.\end{proof}

\begin{example} The separability assumption in Proposition \ref{ct:1} is essential, and for that reason we cannot extend the result to higher derived sets in an obvious way. Indeed, take $\kappa <{\mathfrak c}$ and let $K_0$ be the subset of those elements $x$ of the Cantor cube $2^\kappa$ for which the support  $\{\xi<\kappa: x(\xi)\neq 0\}$ has at most 2 elements.
Take any   almost disjoint family $\c A$ of size $\kappa$ and $K_{\cA}$ as the second compactum.
Then $K_0'=\bA(\kappa)=K_{\c A}'$; however,
$C(K_0)\not\simeq C(K_{\c A})$ since $C(K_0)$ is weakly compactly generated (in other words, $K_0$ is Eberlein compact) while the latter space is not
($K_{\c A}$ is not Eberlein compact since it is separable,  but not metrizable).
\end{example}

Since each  scattered compact space $K$ of height $2$ is a finite sum of one-point compactifications of discrete spaces, the corresponding function space $C(K)$ is isomorphic to $c_0(|K|)$. Hence, from Theorem \ref{ct:1}  we  obtain the following corollaries:

\begin{corollary}
Assuming ${\sf MA}(\kappa)$, if $K_0$ and $K_1$ are separable scattered compact spaces of height $3$ and weight $\kappa$, then $C(K_0)$ and $C(K_1)$ are  isomorphic.
\end{corollary}

\begin{corollary}
Assuming ${\sf MA}(\kappa)$, if $K$ is a separable scattered compact space of height 3 and weight $\kappa$, then $c_0(C(K))$
 (the $c_0$--direct sum of $C(K)$)
 is isomorphic to $C(K)$. In particular, $C(K)$ is isomorphic to its square.
\end{corollary}
\begin{proof} Note that $c_0(C(K))$ is isomorphic to $C(\bA(\omega)\times K)$ and such compact space has finite height, so we  infer from \cite[Corollary 4.2]{CT18}  that $\Ext(c_0(C(K)), c_0)=0$. On the other hand, we have the following diagram
$$\begin{CD}
0@>>> c_0 @>>> c_0(C(K)) @>>> c_0(c_0(|K|))@>>> 0\\
&&&&&&@|\\
0@>>> c_0 @>>> C(K) @>>> c_0(|K|)) @>>> 0
\end{CD}$$
Hence, the diagonal principles (see Proposition \ref{semi}) yield $c_0 \oplus C(K) \sim c_0 \oplus c_0(C(K))$. Since $C(K)$ contains $c_0$ complemented, and so does $c_0(C(K))$, we are done. In particular, $C(K)$ is isomorphic to its square, because $c_0(C(K))$ has this property.
\end{proof}

\section{Counting non-isomorphic $C(K)$-spaces}\label{counting}

For any compact space $K$ we identify, as usual,  the dual space $C(K)^\ast$ with the space $M(K)$, of all signed regular Borel measures on $K$ having
finite variation. We start by  the following general result on the isomorphism types of $C(K)$ spaces.

\begin{theorem}\label{counting:1}
Let $\K$ be a family of compact spaces such that

\begin{enumerate}[(i)]
\item $K$ is separable and $|M(K)|=\con$ for every $K\in \K$;
\item For every pair of distinct $K,L\in\K$ one has $C(K)\simeq C(L)$ and $K,L$ are not homeomorphic.
\end{enumerate}
Then $\K$ is of cardinality at most $\con$.
\end{theorem}

\begin{proof}
Since every $K\in\K$ is separable, there is a continuous surjection $\pi_K:\beta\omega\to K$. Consequently,
$C(K)$ can be isometrically embedded into $C(\beta\omega)$ via the mapping $\pi_K^\ast(g)=g\circ \pi_K$, $g\in C(K)$.

Take $K,L\in\K$, $K\neq L$. Note that if the condition
\[ \pi_K(p_1)=\pi_K(p_2)\iff \pi_L(p_1)=\pi_L(p_2),\]
was satisfied for any $p_1,p_2\in\beta\omega$ then the formula $h(\pi_K(p))=\pi_L(p)$ would properly define
a homeomorphism $h:K\to L$. Hence, for instance,  there are $p_1,p_2\in\beta\omega$ such that
$\pi_K(p_1)=\pi_K(p_2)$ while $\pi_L(p_1)\neq \pi_L(p_2)$. This immediately implies that
$\pi_K^\ast[C(K)]\neq \pi_L^\ast[C(L)]$ --- consider a function $g\in C(L)$ that distinguishes $\pi_L(p_1)$ and $\pi_L(p_2)$.

Fix now $K_0\in\K$; for every $K\in\K$ there is an isomorphism $T_K:C(K_0)\to C(K)$ and therefore
$S_K:C(K_0)\to C(\beta\omega)$, where $S_K=\pi_K^\ast\circ T_K$ is a bounded linear operator. It follows that
$S_K\neq S_{K'}$ whenever $K\neq K'$. To conclude the proof, it is therefore  sufficient to show that the space
$\cL(C(K_0), C(\beta\omega))$ of all bounded operators, is of size at most $ \con$.

Note that any operator $R: C(K_0)\to C(\beta\omega)$ is uniquely determined by the sequence
\[\la R^\ast(\delta_n): n\in\omega\ra,\]
of measures from $M(K_0)$,
where
$R^\ast: M(\beta\omega)\to M(K_0)$ is the adjoint operator. By our assumption, $|M(K_0)|=\con$ so $|M(K_0)|^\omega=\con$,
and we are done.
\end{proof}

Let us remark that, in the setting of the above theorem, if we know only that $|M(K_0)|=\con$ for some $K_0\in\K$, then automatically
$|M(K)|=\con$ for every $K\in\K$ (since $C(K)\simeq C(K_0)$). However, separability of the domain is not preserved by isomorphisms between the spaces
of continuous functions. For instance, $\ell_\infty\simeq L_\infty[0,1]$ by Pe{\l}czy\'nski theorem, so $C(\beta\omega)\simeq C(K)$, where
$K$ is the Stone space of the measure algebra (recall that such $K$ is not separable).

\begin{corollary}\label{counting:2}
There are $2^\con$ pairwise nonisomorphic twisted sums of $c_0$ and $c_0(\con)$.
\end{corollary}

\begin{proof}
In what follows, $\cA, \cB$ (with possible indices) denote almost disjoint families of subsets of $\omega$ of cardinality $\con$.
For every almost disjoint family $\cA$, $K_\cA$ is a separable compact space and, since $K_\cA$ is scattered, $|M(K_\cA)|=\con$.

Note first that any homeomorphism $h: K_\cA\to K_\cB$ is  determined by a permutation of $\omega$. Therefore, we can define
a sequence $\la \cA_\xi: \xi<2^\con\ra$ such that the spaces $K_{\cA_\xi}$ are pairwise not homeomorphic.
By a direct application of Theorem \ref{counting:1}, we conclude that there is $I\sub2^\con$ of cardinality $2^\con$ such that
$C(K_{A_\xi})$ is not isomorphic to $C(K_{A_\eta})$ whenever $\xi,\eta\in I$ and $\xi\neq \eta$.
Now, the assertion follows from Theorem \ref{tcc:1}.
\end{proof}

The argument that proves Theorem \ref{counting:1} and Corollary \ref{counting:2} can be generalized to some higher cardinals, see Section  \ref{appendix}.

Recall that, given  a scattered compact space $K$,  we have the equality $w(K) = |K|$ (see \cite{So03}).
Since, for infinite compacta $K$, weights of $K$ and $C(K)$ are equal, we obtain the following simple observation.

\begin{prop}
Let $\cA$ and $\cB$ be infinite almost disjoint families of subsets of $\omega$. If $C(K_\cA)$ and $C(K_\cB)$ are isomorphic then $|\cA| = |\cB|$.
\end{prop}

Using the construction from the proof of Theorem 8.7 from \cite{maple}, we shall now define a large family of compactifications of $\omega$ with remainders homeomorphic to $\bS$. We follow here the notation introduced in
subsection \ref{das}; recall, in particular,  that for ane $x\in (0,1)$ we write $P_x$ for $\{q\in Q: q\le x\}$.

For every $x\in (0,1)$, fix  a nondecreasing  sequence $(q_x^n)_{n\in\omega}$ in $Q$ such that $\lim_{n}q_x^n = x,\ q_x^n<x$, and
put  $S_x=\{q_x^n: n\in\omega\}$. Take any function $\theta : (0,1) \to 2$ and define
\begin{equation}
\label{ltheta}  R_x^\theta =\begin{cases} P_x& \text{if $\theta(x)=0$},\\ P_x\setminus S_x & \text{if $\theta(x)=1$}.
\end{cases}
\end{equation}

Then $R_x^\theta\sub^\ast R_y^\theta$ whenever $x<y$ since, in such a case,  $P_x\cap S_y$ is finite.

Let $\mathfrak{B}^\theta$ be a subalgebra of ${\mathcal P}(Q)$ generated by $\{R_x^\theta: x\in (0,1)\}\cup\fin$,
where $\fin$ denotes the family of all finite subsets of $Q$; we let  the space $L^\theta$ be $\ult(\fB^\theta)$, the Stone space
of the underlying Boolean algebra.

Note that every space $L^\theta$ may be seen as a compactification of the discrete set $Q$ with the remainder homeomorphic to $\bS$.
Indeed, we may think that $Q\sub L^\theta$ by identifying $q\in Q$ with $\cF\in\ult(\fB^\theta)$ containing $\{q\}$. If $\cF\in \ult(\fB^\theta)$ contains no finite
subset of $Q$ then it is uniquely determined by the set $\{x\in (0,1); R_x\in\cF\}$. Now the point is that
the family $\{ R_x:x\in (0,1)\}$ forms a chain with respect to  almost inclusion. Hence the Boolean algebra $\fB/\fin$ is isomorphic to $\fA$ which means that
$L^\theta\setminus Q$ is homeomorphic to $\ult (\fA) \simeq \bS$ (see also \cite[Theorem 8.7]{maple}).
It follows that every space $C(L^\theta)$ is a twisted sum of $c_0$ and $C(\bS)$..

\begin{corollary}\label{counting:3}
There are $2^\con$ pairwise nonisomorphic twisted sums of $c_0$ and $C(\bS)$.
\end{corollary}

\begin{proof}
We can follow the idea of the proof of Corollary \ref{counting:2}. We have a family $\{L^\theta: \theta\in 2^{(0,1)}\}$ of separable compact spaces.
For any $\theta$, $|M(L^\theta)|=\con$ because $L^\theta$ may be identified with $\bS\cup\omega$ and $M(\bS)$ has cardinality $\con$.
The latter follows from the fact that a compact separable linearly ordered space carries at most $\con$ regular measures, see e.g.\
Mercourakis \cite{Me96}.

Again, for  a fixed space $L^\theta$,  the set $\{\eta\in 2^{(0,1)}: L^\theta\simeq L^\eta\}$ is of size at most $\con$ (since any homeomorphism of such spaces
fixes $Q$ and, as before, we can single out  $2^\con$ many pairwise nonisomorphic spaces of the form $C(L^\theta)$.
\end{proof}

\section{Separable Rosenthal compacta}

The purpose of this section is to collect several subtle results concerning applications of descriptive set theory to separable Rosenthal compacta  that will be needed in the next section.

A compact space $K$ is {\em Rosenthal compact} if embeds into $B_1(\omega^\omega)$, the space of Baire-one functions on the Polish space $\omega^\omega$ (homeomorphic to the irrationals), equipped with the topology of pointwise convergence; recall that a Baire-one function is a pointwise limit of a sequence of continuous functions.
We refer to \cite{De14} and \cite{Mar03} for basic properties
of Rosenthal compacta and further references.

Recall that in a Polish space $T$, the Borel $\sigma$-algebra $Bor(T)$ can be written as
\[ Bor(T)=\bigcup_{1\le \alpha<\omega_1}\Sigma^0_\alpha(T)=
 \bigcup_{1\le \alpha<\omega_1}\Pi^0_\alpha(T),\]
 where $\Sigma^0_1(T)$ and $\Pi^0_1$ are the families of open and closed sets, respectively. Additive classes $\Sigma^0_\alpha$ are defined inductively as
 all countable unions of elements from $\bigcup_{\beta<\alpha} \Pi^0_\beta$ and so on, see
Kechris \cite[11.B]{Ke95} for details. Also recall that a subset of a Polish space is {\em analytic} if it is a continuous image of $\omega^\omega$. In a Polish space, every Borel set is analytic.

\par If $K$ is any separable compact space, $D\sub K$ is any countable dense subset and $X$ is a topological space, then we denote by $C_D(K,X)$ the space of all continuous functions from $K$ into $X$, equipped with the topology of pointwise convergence on $D$. We write $C_D(K)$ for $C_D(K,\er)$. In other words, $C_D(K)$ is a subset of $\er^D$ which is  the image of the restriction map
$C(K)\ni g\to g_{|D}\in \er^D$.

Recall that the dual Banach space $C(K)^\ast$ may be identified with the space $M(K)$,  of regular Borel measures on $K$ of finite variation.
We extend the notation introduced above as follows. Consider any countable subset $M\sub M(K)$ separating the functions in $C(K)$; then
$C_M(K)$ stands for the space $C(K)$ endowed with
the weak topology generated by $M$. Here, we simply  identify $f\in C_M(K)$ with $(\mu(f))_{\mu\in M}\in \er^M$.

We briefly mention some properties of Rosenthal compacta. Godefroy showed in \cite{Go80} that if $K$ is Rosenthal compact, then $M_1(K)$, the dual unit ball, is Rosenthal compact in its $weak^\ast$ topology.  On the other hand, Bourgain, Fremlin and Talagrand \cite{BFT78} proved that every Rosenthal compactum $K$ is {\em Fr\'echet-Urysohn}, i.e., for any $A\subseteq K$ and
$x\in \overline{A}$, there is a sequence $(x_n)_{n\in\omega}$ of points from $A$ which converges to $x$. Those properties, in particular, imply
that if $D\sub K$ is a countable dense set then every $\mu\in M(K)$ is a Baire-one function on $C_D(K)$; see the proof of Theorem 3.1 in \cite{Ma88}.

\par The following result can be used as a characterization of separable Rosenthal compacta; part $(a)$ of Theorem \ref{src:1} is due to Godefroy \cite{Go80}, while part $(b)$ is Corollary 2.4 in Dobrowolski and Marciszewski \cite{DM95}.

\begin{theorem}\label{src:1}
\begin{enumerate}[(a)]
\item A separable compact space $K$ is Rosenthal compact if and only if $C_D(K)$ is analytic for every countable dense set $D\sub K$.
\item If $K$ is a compact space and $D\sub K$ is a countable dense set such that $C_D(K)$ is analytic then
either $K$ is Rosenthal compact or $K$ contains a copy of $\beta\omega$.
\end{enumerate}
\end{theorem}

\par Godefroy's characterization of separable Rosenthal compacta
was followed by introducing in \cite{Ma88} a certain index measuring the complexity of
 such spaces.
This kind of Rosenthal's index will be denoted here by $\mi(\cdot)$; note that in \cite{Ma88} the working notation $\eta(\cdot )$ was used.

\begin{defin}\label{src:2} We define the index $\mi$ on the class of separable Rosenthal compacta as follows.
Set $\mi(K)=\omega_1$ if $C_D(K)$ is Borel in $\er^D$ for no countable dense set $D\sub K$. Otherwise, set $\mi(K)=\alpha$, where
$\alpha$ is the least ordinal number $<\omega_1$ such that
\[ C_D(K) \in \Sigma^0_{1+\alpha}(\er^D) \cup \Pi^0_{1+\alpha}(\er^D),\]
for some countable dense $D\sub K$.
\end{defin}

The reader should be warned that the difference between Definition \ref{src:2} and that from \cite{Ma88} is connected with the fact that older tradition was to
count Borel classes starting from 0 rather than 1 (for instance, $F_{\sigma\delta}$ is $\Pi^0_3$).
Recall that $\mi(K)\ge 2$ whenever $K$ is infinite, see \cite[Theorem 2.1]{Ma88}. The double arrow space $\bS$ is a classical nonmetrizable Rosenthal compactum
with $\mi(\bS)=2$.

The main feature of the index $\mi$ is that it is almost preserved by isomorphisms of Banach spaces, as we show next.
The following result is a particular case of \cite[Corollary 3.2]{Ma88};
We outline the main idea of its proof because the argument  is a much shorter in our setting.

\begin{theorem}\label{src:3}
If $K$ and $K'$ are separable Rosenthal compacta and $C(K)\simeq C(K')$ then
\[ \mi(K)\le 1+\mi(K') \mbox{ (and, by symmetry,  } \mi(K')\le 1+\mi(K)).\]
In particular, $\mi(K)=\mi(K')$ whenever  $\mi(K)\ge\omega$.
\end{theorem}

\begin{proof} Let $T:C(K)\to C(K')$ be an isomorphism such that $c\cdot \|g\|\le \|Tg\|\le \| g\|$ for every $g\in C(K)$, where $c>0$.
 Fix countable dense sets $D\sub K$ and $D'\sub K'$ realizing the values of $\mi(K)$ and $\mi(K')$, respectively, and write $\Delta_D=\{\delta_d: d\in D\}$,
$\Delta_{D'}=\{\delta_d: d\in D'\}$. \medskip

\par \textbf{Claim 1.} \emph{ There are countable sets $M, M'$, where $\Delta_D\sub M\sub M_1(K)$, $\Delta_{D'}\sub M'\sub M_1(K')$ such that
\[ C_M(K)\ni g\to Tg\in C_{M'}(K'),\]
is a homeomorphism.}

\begin{proof}[Proof of the claim:] Note that $T^\ast$ sends $M_1(K')$ into $M_1(K)$. Likewise, if we consider $S:C(K')\to C(K)$ given by $S=c\cdot T^{-1}$ then $S^\ast$ sends $M_1(K)$ into
$M_1(K')$. To define $M$ and $M'$ put $M(0)=\Delta_D$ and $M'(0)=\Delta_{D'}$ and define inductively
\[ M_{n+1}=T^\ast[M'(n)],\quad M'(n+1)=S^\ast[M(n)].\]
Then the sets $M=\bigcup_n M(n)$ and $M'=\bigcup_n M'(n)$ are as required. For instance, suppose that a sequence $(g_k)_k$ converges in $C_M(K)$ and
consider any $\nu\in M'$. Then $\nu\in M'(n)$ for some $n$ and therefore $\mu=T^\ast\nu\in M_{n+1}\sub M$.
Hence $\nu(Tg_k)=T^\ast \nu (g_k)\to T^\ast\nu (g)=\nu(Tg)$. \end{proof}

Now we can examine the mapping $\vf: C_D(K)\to C_{D'}(K')$ closing the following diagram

$$\xymatrix{
C_M(K)\ar[rr]^T\ar[d]^{id}&&C_{M'}(K')\ar[d]^{id}\\
C_{D}(K)\ar@{.>}[rr]^\vf   &&C_{D'}(K')
}$$

\par \textbf{Claim 2.} \emph{$\vf$ and $\vf^{-1}$ are mapping of the first Baire class.}
\begin{proof}[Proof of the claim:]  Indeed, while the mapping $id: C_M(K)\to C_D(K)$ is continuous (as $\Delta_D\sub M)$, its inverse is of the first Baire class since every $\mu\in M_1(K)$ is a $weak^\ast$ limit of a sequence from the absolute convex hull of $\Delta_D$. This shows that $\vf$ is of the first Baire class; the argument for $\vf^{-1}$ is symmetric. \end{proof}

The final step is to use Kuratowski's theorem \cite[par.\ 35 VII]{Ku64} which assures that there are $F_{\sigma\delta}$ sets $A,B$ containing $C_D(K)$ and $C_{D'}(K')$ respectively, and an extension $\wt{\vf}$ of $\vf$ to a Baire-one isomorphism $A\to B$. This, together with $\mi(K),\mi(K')\ge 2$ gives $\mi(K)\le 1+\mi(K')$
and $\mi(K')\le 1+\mi(K)$.
\end{proof}

\section{Twisted sums and Rosenthal compacta }\label{tsr}
Denote by $2^{<\omega}$ the full dyadic tree and, as usual, $2^\omega$ is the Cantor set. For any $x\in 2^\omega$, set
 $B(x)=\{x|{n}: n\in \omega\} \subset 2^{<\omega}$. It is clear that for any $Z\sub 2^\omega$,
 \begin{equation}
 \label{caz} \cA_Z=\{B(x):x\in Z\},
 \end{equation}
 is an almost disjoint family of infinite subsets of $2^{<\omega}$.
\par The mapping $x\mapsto B(x)$ is a homeomorphic embedding of $2^\omega$ into $2^{2^{<\omega}}$ (we identify $\mathcal{P}(2^{<\omega})$ with $2^{2^{<\omega}}$). Therefore, for any $Z\sub 2^\omega$, $\cA_Z$ is an almost disjoint family (in $2^{2^{<\omega}}$) which is homeomorphic to $Z$. Moreover, for Borel $Z$, we have the following (see \cite[4.2]{Ma88}).

 \begin{theorem}\label{tsr:1}
 If $Z\in \Sigma^0_{1+\alpha}(2^\omega)$, where  $\alpha \ge 2$, and $\cA_Z$ is an almost disjoint family given by (\ref{caz}),
 then $K_{\cA_Z}$ is Rosenthal compact and $\alpha\le\mi(K_{\cA_Z})\le 1+\alpha+1$.
 \end{theorem}

 Using Theorem \ref{tsr:1} and Theorem \ref{src:3} we arrive at the following.

\begin{corollary}\label{tsr:2}
There is a family $\{ K_\xi:\xi<\omega_1\}$ of separable Rosenthal compacta such that $C(K_\xi)\not\simeq C(K_\eta)$ whenever $\xi\neq\eta$ and
every $C(K_\xi)$ is a (nontrivial) twisted sum of $c_0$ and $c_0(\con)$.
\end{corollary}

\begin{prop}\label{tsr:5}
Let
$$0\longrightarrow c_0\longrightarrow C(L)\longrightarrow C(K)\longrightarrow 0$$
be a twisted sum, where $K$ is an infinite separable Rosenthal compact space, and $L$ is a compact space. If this twisted sum is trivial, then $L$ is a separable Rosenthal compact space, and $\mi(L)\le 1+\mi(K)$.
\end{prop}

\begin{proof}
Triviality of our twisted sum gives us the following string of isomorphisms
$$C(L)\simeq c_0\oplus C(K)\simeq C(\omega+1)\oplus C(K)\simeq C((\omega+1)\oplus K)\,.$$
Godefroy \cite{Go80} proved that the class of separable Rosenthal compacta is preserved by isomorphisms of function spaces; since $(\omega+1)\oplus K$ belongs to this class, so does $L$. By Theorem \ref{src:3} it is enough to verify that $\mi((\omega+1)\oplus K)\le \mi(K)$. Let $D$ be a countable dense subset of $K$ realizing the value of $\mi(K)$. By \cite[Theorem 2.1]{Ma88} $C_D(K)$ is not a $G_{\delta\sigma}$-subset of $\er^D$, and it is well known that $C_\omega(\omega+1)$ is an $F_{\sigma\delta}$-subset of $\er^\omega$. Hence, for the countable dense subset $E=\omega\cup D$ of $(\omega+1)\oplus K$, the space $C_E((\omega+1)\oplus K)$ can be identified with the product $C_\omega(\omega+1)\times C_D(K)$ which is a Borel subset of $\er^\omega\times \er^D$ of the class $\mi(K)$.
\end{proof}

We turn to examining twisted sums of $c_0$ and $C(\bS)$; below we follow the notation introduced in \ref{das} and Section \ref{counting};
recall that the space $L^\theta$ was defined just below formula (\ref{ltheta}).
Here the parameter $\theta$ is any function $\theta: (0,1)\to 2$. It will be now convenient to write $\theta$ as $\chi_Z$,
the characteristic function of a set $Z\sub (0,1)$.
In such a case, we write simply $L(Z)$ for $L^\theta$ with $\theta=\chi_Z$.

\begin{lemma}\label{tsr:4}
For any $Z\subseteq (0,1)$, the space $C_Q(L(Z))$ contains a $G_\delta$-subset $X$ homeomorphic to $[0,1]\setminus (Z\cup Q)$.
\end{lemma}

\begin{proof}
We will look for $X$ inside the closed subset $C_Q(L(Z),2) = C_Q(L(Z))\cap 2^Q$ of $C_Q(L(Z))$, which may be seen as traces of clopen subsets of
$L(Z)$ on $Q$, i.e., the algebra $\mathfrak{B}^\theta$. Define
\begin{eqnarray*}
P_1 &=& \{(x,1)\cap Q: x\in [0,1]\setminus Q\}\subseteq 2^Q\\
P_2 &=& \{(x,1)\cap Q: x\in Q\}\subseteq 2^Q\\
P_3 &=& \{[x,1)\cap Q: x\in Q\}\subseteq 2^Q
\end{eqnarray*}
where we identify $\mathcal{P}(Q)$ with $2^Q$. Observe that the union $P= P_1\cup P_2\cup P_3$ is closed in $2^Q$. Indeed, we have $s\in 2^Q\setminus P$ if and only if there exists $p<q$ in $Q$ with $s(p)= 1$ and $s(q) = 0$. Since $P_2\cup P_3$ is countable, $P_1$ is $G_\delta$-subset of $2^Q$. A routine verification shows that the mapping $x\mapsto (x,1)\cap Q$ is a homeomorphism of $[0,1]\setminus Q$ onto $P_1$. Analyzing the description of the algebra $\mathfrak{B}^{\chi_Z}$ defining
$L(Z)$ one can verify that $[(x,1)\cap Q]\in P_1\cap C_Q(L(Z))$ if and only if $x\in [0,1]\setminus (Z\cup Q)$. Hence, the set $X = P_1\cap C_Q(L(Z))$ has the required properties.
\end{proof}

\begin{corollary}\label{tsr:4.5}
Let $Z\subseteq (0,1)$ be such that $Z\notin \Sigma^0_{5}((0,1)) \cup \Pi^0_{5}((0,1))$. Then $C(L(Z))$
is a nontrivial twisted sum of $c_0$ and $C(\bS)$.
\end{corollary}

\begin{proof}
Suppose that $C(L(Z))$ represents a trivial twisted sum of $c_0$ and $C(\bS)$. Since $\mi(\bS)=2$, from Proposition \ref{tsr:5}
we consclude that $L(Z)$ is Rosenthal compact with $\mi(L(Z))\le 3$. Theorem 2.2 from \cite{Ma88} says that, for a separable Rosenthal compact space $K$, and any two countable dense subsets $D,E$ of $K$, the Borel classes of $C_D(K)$ and $C_E(K)$ can differ by at most $1$.
Therefore $C_Q(L(Z))\in \Sigma^0_{5}((\er^Q) \cup \Pi^0_{5}((\er^Q)$. Hence, any $G_\delta$--subset of
$C_Q(L(Z))$ is also Borel of the same class, a contradiction with Lemma \ref{tsr:4} and our assumption on $Z$.
\end{proof}

To prove the next Lemma \ref{tsr:3} we need to employ a more effective description of the sets $S_x$ used in Section 5.
We now consider  $Q={\mathbb Q} \cap (0,1)$;
for $x\in (0,1)$, let $0.i^x_0i^x_1\dots$ be the binary expansion of $x$ using infinitely many $1'$s, i.e., $i^x_k\in\{0,1\}, k\in\omega$ and $x= \sum_{k=0}^{\infty}(i^x_k/2^{k+1})$. We define
$$S_x =\left\{\sum_{k=0}^{n} i^x_{k}/2^{k+1}: n\in\omega\right\}\setminus \{0\}.$$
Using the fact that, for $x\in (0,1)\setminus Q$, the expansion $0.i^x_0i^x_1\dots$ also has infinitely many $0'$s, one can easily verify that the mapping $x\mapsto S_x$ is continuous on $(0,1)\setminus Q$ (actually, it is a homeomorphic embedding of $(0,1)\setminus Q$ into $2^Q$).

We also need the following two auxiliary results.
The first one is well known, and follows easily from the fact that Boolean operations on $\mathcal{P}(\omega)$ are continuous.

\begin{prop}\label{tsr:6}
Each subalgebra of $\mathcal{P}(\omega)$ with an analytic set of generators is analytic.
\end{prop}

The second result is probably also well known, it can be derived from the results of Godefroy \cite{Go80}. For the sake of completeness we include an (alternative) proof of it.

\begin{prop}\label{tsr:7}
Let $K$ be a separable zero-dimensional compact space and $D$ a countable dense subspace of $K$. If $C_D(K,2)=C_D(K)\cap 2^D$ is analytic, so is $C_D(K)$.
\end{prop}

\begin{proof}
Since $\er$ is homeomorphic to $(0,1)$ and $C_D(K,(0,1)) = C_D(K,[0,1])\cap (0,1)^D$ is a $G_\delta$-subset of $C_D(K,[0,1])$, it is enough to prove that $C_D(K,[0,1])$ is analytic.
The countable product $C_D(K,2)^\omega$ can be identified with the space $C_D(K,2^\omega)$. Hence, $C_D(K,2^\omega)$ is also analytic, and it is enough to show that $C_D(K,[0,1])$ is a continuous image of $C_D(K,2^\omega)$. By \cite[Lemma 1]{Okunev}, whose key ingredient is the Marde\v{s}ic factorization theorem \cite{Mardesic}, there exists a continuous map $\phi: 2^\omega\to [0,1]$ such that the composition map $\phi^\circ: C(K,2^\omega)\longrightarrow C(K,[0,1])$ given by $g\to \phi\circ g$ is surjective.\end{proof}

\begin{lemma}\label{tsr:3}
If $\theta:(0,1)\to 2$ is Borel then $L^\theta$ is Rosenthal compact.
\end{lemma}

\begin{proof}
Let $\theta = \chi_Z$, obviously $Z$ is a Borel subset of $(0,1)$. We will consider the Borel sets $Z_0 = Z\setminus Q$ and $Z_1 = (0,1)\setminus(Z\cup Q)$.

First, we will show that the set $\{R^\theta_x: x\in (0,1)\}\subset \mathcal{P}(Q)$ of generators of the algebra $\mathfrak{B}^\theta$ is analytic. We can neglect the countable set $\{R^\theta_x: x\in Q\}$, so it is enough to verify that both sets $G_i = \{R^\theta_x: x\in Z_i\},\ i=0,1$, are analytic.
We have $G_0 = \{P_x\setminus S_x: x\in Z_0\}$ and $G_1 = \{P_x: x\in Z_1\}$.
Since the map $A\mapsto Q\setminus A$ is a homeomorphism of $\mathcal{P}(Q)$ and $Q\setminus P_x = (x,1)\cap Q$, the argument from the proof of Lemma \ref{tsr:4} shows that $G_1$ is homeomorphic to $Z_1$, hence analytic. Observe that, by the same argument and the continuity of the mapping $x\mapsto S_x$, the map $\varphi: Z_0\to G_0$ defined by $\varphi(x) = P_x\setminus S_x$ is continuous and onto, therefore $G_0$ is also  analytic.

Lemma \ref{tsr:6} implies that the algebra $\mathfrak{B}^\theta$, which can be identified with $C_Q(L^\theta,2)$, is analytic. In turn, from Lemma \ref{tsr:7} we infer that $C_Q(L^\theta)$ is also analytic. Now, Theorem \ref{src:1}(b) gives us the desired conclusion.
\end{proof}

\begin{corollary}
If the set $Z\sub (0,1)$ is not coanalytic then
\begin{enumerate}[(i)]
\item $L(Z)$ is not Rosenthal compact;
\item  $C(L(Z))$ is a nontrivial twisted sum of $c_0$ and $C(\bS)$;
\item  $C(L(Z))$ is not isomorphic to $C(L^\theta)$ whenever the function $\theta$ is
Borel.
\end{enumerate}
\end{corollary}

\begin{proof}
Part $(i)$ follows directly from  Lemma \ref{tsr:4}; $(ii)$ is a consequence of $(i)$ and Proposition \ref{tsr:5}.
The last statement follows from Lemma \ref{tsr:3}.
\end{proof}

It is likely that twisted sums of $c_0$ and $C(\bS)$  that are of the form $C(K)$ with $K$ Rosenthal compact can be examined more closely using the
Rosenthal index. However, the following problem is open to us.

\begin{prob}
Let $\theta:(0,1)\to 2$ be Borel. Is $\mi(L^\theta)<\omega_1$? Can we find an effective estimate of $\mi(L^\theta)$ using the class of $\theta$?
\end{prob}

\section{Counting non-isomorphic $C(K)$-spaces one more time} \label{appendix}

We  consider here families $\cA$ consisting of countable infinite subsets of an uncountable set $I$, as in Dow and Vaughan \cite{DV10}.
Such $\cA$ is  an almost disjoint family if, again,  $A\cap B$ is finite for any distinct  $A,B\in \cA$.
We can form
a version of Aleksandrov-Urysohn space $K_\cA$, where
$$K_{\cA} = I \cup \{A: A\in \cA\} \cup \{\infty\}.$$
As before, points in $I$ are isolated, basic open neighborhoods of a given point $A\in\cA$ are of the form $\{A\}\cup (A\setminus  F)$ with $F\subseteq I$ finite, and
the point $\infty$  compactifies   the locally compact space $I\cup \{A:A\in\c A\}$.
Then $K_\cA$ is a scattered compact space of height $3$ and density $|I|$.

The arguments used in Theorem \ref{counting:1} and Corollary \ref{counting:2} can be generalized to obtain information about the cardinality of
$\Ext(c_0(\kappa^\omega), c_0(\kappa))$ for $\kappa$ satisfying  $\kappa<\kappa^\omega$ --- we outline it here.

\begin{theorem} Assume that $\kappa<\kappa^\omega$ and let $\cA$ be an almost disjoint family of countable subsets of $\kappa$ with $|\cA|=\kappa^\omega$.
 Then $C(K_\cA)$ is a non-trivial twisted sum of $c_0(\kappa)$ and $c_0(\kappa^\omega)$.
\end{theorem}

\begin{proof}
Write $K_\cA'$ for the derivative of $K_\cA$, which is equal to $\cA\cup\{\infty\}$. It is enough to check that
there is no bounded extension operator $E: C(K_\cA') \to C(K_\cA)$.
Assume the existence of such $E$;  for any $A\in \cA$, the singleton $\{A\}$ is open in $K_\cA'$ so, writing $g_A$ for its characteristic function,
we have $g_A\in C(K_\cA')$.
Hence,
there must be $\xi_A\in \kappa$ so that
$|Eg_A(\xi_A)|>\frac12$. Since $|\cA|=\kappa^\omega>\kappa$,
we infer the existence of a point $\xi\in \kappa$ such that $\xi=\xi_{A_n}$ for a sequence of distinct $A_n\in\cA$.
Then, for every natural number  $m$ we obtain
$$\left|\sum_{j=1}^m Eg_{A_j}(\xi)\right| \geq \frac m2 \quad \mbox{and} \quad \left\|\sum_{j=1}^m g_{A_j}\right\|=1,$$
which contradicts continuity of $E$.
\end{proof}

The next results can be proved by similar arguments to those  used in \ref{counting:1} and \ref{counting:2}.

\begin{theorem}
Fix an infinite cardinal $\kappa$, and let $\K$ be a family of compact spaces such that
\begin{enumerate}[(i)]
\item every $K$ has density $\kappa$ and $|M(K)|\le 2^\kappa$;
\item For every pair of distinct $K,L\in\K$ one has $C(K)\simeq C(L)$ and $K, L$ are not homeomorphic.
\end{enumerate}
Then $\K$ is of cardinality at most $2^\kappa$.
\end{theorem}

\begin{corollary} If  $2^{\kappa^\omega}>2^\kappa$ then there are $2^{\kappa^\omega}$ pairwise non-isomorphic twisted sums of $c_0(\kappa)$ and $c_0(\kappa^\omega)$.
\end{corollary}

The above corollary can be applied to any (infinite) $\kappa<\con$ under Martin's axiom since then $2^\kappa= \con$.

\end{document}